\documentclass[times,12pt]{elsarticle}
\journal{}

\makeatletter
\def\ps@pprintTitle{%
 \let\@oddhead\@empty
 \let\@evenhead\@empty
 \def\@oddfoot{\centerline{\thepage}}%
 \let\@evenfoot\@oddfoot}
\makeatother

\usepackage{amssymb,amsthm,mathtools}
\usepackage[colorlinks]{hyperref}

\DeclareRobustCommand{\brkbinom}{\genfrac\{\}{0pt}{}}
\usepackage{xcolor,dsfont,geometry}
\newtheorem{theorem}{Theorem}
\newtheorem{lemma}{Lemma}
\newtheorem{remark}{Remark}

\makeatletter \@addtoreset{equation}{section} \makeatother

\newcommand{\N}{\mathbb{N}}
\newcommand{\R}{\mathbb{R}}
\newcommand{\EE}{\mathsf{E}} 
\newcommand{\bb}[1]{\boldsymbol{#1}}

\begin{document}

\begin{frontmatter}

\title{Central and noncentral moments of the \\multivariate hypergeometric distribution}%

\author[a1,a2]{Fr\'ed\'eric Ouimet}%

\address[a1]{Department of Mathematics and Statistics, McGill University, Montr\'eal (Qu\'ebec) Canada H3A 0B9}%

\ead{frederic.ouimet2@mcgill.ca}%

\begin{abstract}
In this short note, explicit formulas are developed for the central and noncentral moments of the multivariate hypergeometric distribution. A numerical implementation is provided in \texttt{Mathematica} for fast evaluations. This work complements the paper by \citet{Ouimet_2021_multinomial_moments}, where analogous formulas were derived and implemented in \texttt{Mathematica} for the multinomial distribution.
\end{abstract}

\begin{keyword}
hypergeometric distribution \sep moments \sep multinomial distribution \sep multivariate hypergeometric distribution \sep sampling without replacement
\MSC[2020]{Primary: 62E15 Secondary: 60E05, 62H10}
\end{keyword}

\end{frontmatter}

\section{Introduction}\label{sec:intro}

For any integer $d\in \N$, define the $d$-dimensional (unit) simplex by
\[
\mathcal{S}_d = \big\{\bb{x}\in [0,1]^d: \|\bb{x}\|_1 \leq 1\big\},
\]
where $\|\bb{x}\|_1 = \sum_{i=1}^d |x_i|$ denotes the $\ell^1$ norm on $\R^d$.
Given $d+1$ disjoint subpopulations of respective size $N_1,\ldots,N_{d+1}$, assume that $N_1 + \dots N_{d+1} = N$, where $N$ denotes the total population count. Consider the random vector $(\bb{X},X_{d+1}) \equiv (X_1,\ldots,X_d,X_{d+1})$ which represents the vector of subpopulation sample counts when sorting a random sample of size $X_1 + \dots + X_{d+1} \equiv n$, taken without replacement. In that case, the vector of the first $d$ components, $\bb{X} = (X_1,\ldots,X_d)$, is said to follow the $\mathrm{Hypergeometric}_d\hspace{0.2mm}(N,n,\bb{N})$ distribution, written $\bb{X}\sim \mathrm{Hypergeometric}_d\hspace{0.2mm}(N,n,\bb{N})$ for short, where $\bb{N} = (N_1,\ldots,N_d)$ and $N_{d+1} = N - \|\bb{N}\|_1$. The probability mass function of $\bb{X}$ is defined by
\begin{equation}\label{eq:hypergeometric.pmf}
P_{N,n,\bb{N}}(\bb{k}) = \binom{N}{n}^{-1} \prod_{i=1}^{d+1} \binom{N_i}{k_i}, \quad \bb{k}\in \mathbb{K}_d,
\end{equation}
see, e.g., \citet[Chapter~39]{MR1429617}, where $k_{d+1} = n - \|\bb{k}\|_1$, $n,N\in \N$ with $n \leq N$, and the support is
\[
\mathbb{K}_d = \big\{\bb{k}\in \N_0^d \cap n \mathcal{S}_d : k_i\in [0,N_i] ~ \text{for all } i\in \{1,\ldots,d\}\big\}.
\]
When $N = \infty$ (the sampling is done with replacement), the above becomes the multinomial distribution.

Alternatively, one can write $N_i \equiv N p_i$, or equivalently $p_i = N_i / N$, for some probabilities $p_1,\ldots,p_{d+1}$ satisfying $\bb{p} = (p_1,\ldots,p_d)\in \mathcal{S}_d$, where $p_{d+1} = 1 - \|\bb{p}\|_1$, so that the proportion of the $i$-th subpopulation with respect to the total population is represented by $p_i$. The parameter $p_i$ also corresponds to the probability that a specific sample unit be taken from the $i$-th subpopulation. Under this notation, the multivariate hypergeometric distribution is parametrized by $N,n,\bb{p}$ instead of $N,n,\bb{N}$.

\section{Preliminary result}

For any integer $\alpha\in \N_0$ and any real $x\in \R$, let $x^{(\alpha)} = x (x - 1) \dots (x - \alpha + 1)$ denote the falling factorial of order $\alpha$ for $x$, where $x^{(0)} = 1$ by convention. The lemma below derives the joint factorial moments of the multivariate hypergeometric distribution.

\begin{lemma}\label{lem:factorial.moments}
For all $\bb{\alpha}\in \mathbb{K}_d$, one has
\[
\EE\left(\prod_{i=1}^d X_i^{(\alpha_i)}\right) = \frac{n^{(\|\bb{\alpha}\|_1)}}{N^{(\|\bb{\alpha}\|_1)}} \prod_{i=1}^d N_i^{(\alpha_i)}.
\]
\end{lemma}

\begin{remark}
This result can be found without proof for example in Eq.~(39.6) of \citet{MR1429617}; the proof is included here for completeness.
\end{remark}

\begin{remark}
Using the alternative parametrization $N_i \equiv N p_i$, one can write
\[
\EE\left(\prod_{i=1}^d X_i^{(\alpha_i)}\right) = \frac{\prod_{i=1}^d (N p_i)^{(\alpha_i)}}{N^{(\|\bb{\alpha}\|_1)} \prod_{i=1}^d p_i^{\alpha_i}} \times n^{(\|\bb{\alpha}\|_1)} \prod_{i=1}^d p_i^{\alpha_i}.
\]
The second factor on the right-hand side is equal to the corresponding multinomial moment ($N = \infty$), as calculated by \citet{MR143299}. The first factor on the right-hand side represents the correction that accounts for the fact that the samples come without replacement from a finite population.
\end{remark}

\begin{proof}[Proof of Lemma~\ref{lem:factorial.moments}]
For all $\bb{\alpha}\in \mathbb{K}_d$, define
\[
\mathbb{K}_d[\bb{\alpha}] = \big\{\bb{k}\in \N_0^d \cap n \mathcal{S}_d : k_i\in [\alpha_i,N_i] ~ \text{for all } i\in \{1,\ldots,d\}\big\},
\]
then one has
\[
\begin{aligned}
\EE\left(\prod_{i=1}^d X_i^{(\alpha_i)}\right)
&= \sum_{\bb{k}\in \mathbb{K}_d} \left(\prod_{i=1}^d k_i^{(\alpha_i)}\right) P_{N,n,\bb{\alpha}}(\bb{k}) \\
&= \sum_{\bb{k}\in \mathbb{K}_d[\bb{\alpha}]} \left(\prod_{i=1}^d k_i^{(\alpha_i)}\right) \binom{N}{n}^{-1} \binom{N_{d+1}}{k_{d+1}} \prod_{i=1}^d \binom{N_i}{k_i} \\
&= \sum_{\bb{k}\in \mathbb{K}_d[\bb{\alpha}]} \binom{N}{n}^{-1} \binom{N_{d+1}}{k_{d+1}} \prod_{i=1}^d \left\{N_i^{(\alpha_i)} \binom{N_i - \alpha_i}{k_i - \alpha_i}\right\}.
\end{aligned}
\]
Now, for all $\bb{\alpha}\in \mathbb{K}_d$, let
\[
\mathbb{J}_d[\bb{\alpha}] = \big\{\bb{j}\in \N_0^d \cap (n - \|\bb{\alpha}\|_1) \mathcal{S}_d : j_i\in [0,N_i - \alpha_i] ~ \text{for all } i\in \{1,\ldots,d\}\big\},
\]
and $j_{d+1} = (n - \|\bb{\alpha}\|_1) - \|\bb{j}\|_1$ for any $\bb{j}\in \mathbb{J}_d[\bb{\alpha}]$. The above expectation can be rewritten as
\[
\begin{aligned}
\EE\left(\prod_{i=1}^d X_i^{(\alpha_i)}\right)
&= \sum_{\bb{j}\in \mathbb{J}_d[\bb{\alpha}]} \frac{n^{(\|\bb{\alpha}\|_1)}}{N^{(\|\bb{\alpha}\|_1)}} \binom{N - \|\bb{\alpha}\|_1}{n - \|\bb{\alpha}\|_1}^{-1} \binom{N_{d+1}}{j_{d+1}} \prod_{i=1}^d \left\{N_i^{(\alpha_i)} \binom{N_i - \alpha_i}{j_i}\right\} \\
&= \frac{n^{(\|\bb{\alpha}\|_1)}}{N^{(\|\bb{\alpha}\|_1)}} \prod_{i=1}^d N_i^{(\alpha_i)} \times \sum_{\bb{j}\in \mathbb{J}_d[\bb{\alpha}]} \binom{N - \|\bb{\alpha}\|_1}{n - \|\bb{\alpha}\|_1}^{-1} \binom{N_{d+1}}{j_{d+1}} \prod_{i=1}^d \binom{N_i - \alpha_i}{j_i} \\
&= \frac{n^{(\|\bb{\alpha}\|_1)}}{N^{(\|\bb{\alpha}\|_1)}} \prod_{i=1}^d N_i^{(\alpha_i)} \times 1.
\end{aligned}
\]
This concludes the proof.
\end{proof}

\section{Main results}\label{sec:main.results}

We begin by introducing a general formula of the noncentral moments of the multivariate hypergeometric distribution as shown in \eqref{eq:hypergeometric.pmf}.

\begin{theorem}[Noncentral moments]\label{thm:non.central.moments}
Let $\bb{X}\sim \mathrm{Hypergeometric}_d\hspace{0.2mm}(N,n,\bb{N})$. Then, for all $\bb{\alpha}\in \mathbb{K}_d$, one has
\[
\begin{aligned}
\EE\left(\prod_{i=1}^d X_i^{\alpha_i}\right)
&= \sum_{k_1=0}^{\alpha_1} \dots \sum_{k_d=0}^{\alpha_d} \frac{n^{(\|\bb{k}\|_1)}}{N^{(\|\bb{k}\|_1)}} \prod_{i=1}^d \brkbinom{\alpha_i}{k_i} N_i^{(k_i)} \\
&= \sum_{k_1=0}^{\alpha_1} \dots \sum_{k_d=0}^{\alpha_d} \left\{\frac{\prod_{i=1}^d (N p_i)^{(k_i)}}{N^{(\|\bb{k}\|_1)} \prod_{i=1}^d p_i^{k_i}}\right\} n^{(\|\bb{k}\|_1)} \prod_{i=1}^d \brkbinom{\alpha_i}{k_i} p_i^{k_i},
\end{aligned}
\]
where $\brkbinom{\alpha}{k}$ represents a Stirling number of the second kind, which is defined as the number of ways to divide a set of $p$ objects into $k$ non-empty subsets.
\end{theorem}

\begin{proof}
We have the following well-known relation between the power $\alpha\in \N_0$ of a real $x\in \R$ and its falling factorials:
\[
x^{\alpha} = \sum_{k=0}^p \brkbinom{\alpha}{k} \, x^{(k)}.
\]
See, e.g., \citet[p.~262]{MR1397498}. Apply this identity to every factor $X_i^{\alpha_i}$ in the expectation to obtain
\[
\EE\left(\prod_{i=1}^d X_i^{\alpha_i}\right) = \sum_{k_1=0}^{\alpha_1} \dots \sum_{k_d=0}^{\alpha_d} \brkbinom{\alpha_1}{k_1} \dots \brkbinom{\alpha_d}{k_d} \, \EE\left(\prod_{i=1}^d X_i^{(k_i)}\right),
\]
The conclusion follows from Lemma~\ref{lem:factorial.moments}.
\end{proof}

We are now in a position to establish a general formula for the central moments of the multivariate hypergeometric distribution.

\begin{theorem}[Central moments]\label{thm:central.moments}
Let $\bb{X}\sim \mathrm{Hypergeometric}_d\hspace{0.2mm}(N,n,\bb{N})$. Then, for all $\bb{\alpha}\in \mathbb{K}_d$, one has
\[
\begin{aligned}
\EE\left\{\prod_{i=1}^d (X_i - \EE[X_i])^{\alpha_i}\right\}
&= \sum_{\ell_1=0}^{\alpha_1} \dots \sum_{\ell_d=0}^{\alpha_d} \sum_{k_1=0}^{\ell_1} \dots \sum_{k_d=0}^{\ell_d} \frac{n^{(\|\bb{k}\|_1)}}{N^{(\|\bb{k}\|_1)}} \left(\frac{-n}{N}\right)^{\sum_{i=1}^d (\alpha_i - \ell_i)} \prod_{i=1}^d \binom{\alpha_i}{\ell_i} \brkbinom{\ell_i}{k_i} N_i^{\alpha_i - \ell_i} N_i^{(k_i)} \\
&= \sum_{\ell_1=0}^{\alpha_1} \dots \sum_{\ell_d=0}^{\alpha_d} \sum_{k_1=0}^{\ell_1} \dots \sum_{k_d=0}^{\ell_d} \left\{\frac{\prod_{i=1}^d (N p_i)^{(k_i)}}{N^{(\|\bb{k}\|_1)} \prod_{i=1}^d p_i^{k_i}}\right\} n^{(\|\bb{k}\|_1)} (-n)^{\sum_{i=1}^d (\alpha_i - \ell_i)} \prod_{i=1}^d \binom{\alpha_i}{\ell_i} \brkbinom{\alpha_i}{k_i} p_i^{\alpha_i - \ell_i + k_i},
\end{aligned}
\]
where $\binom{\alpha}{\ell}$ refers to the binomial coefficient given by $\frac{\alpha!}{\ell! (\alpha - \ell)!}$.
\end{theorem}

\begin{proof}
Using the binomial theorem on each factor $(X_i - \EE[X_i])^{\alpha_i}$ in the expectation, and bearing in mind that $\EE[X_i] = n N_i / N$ for all $i\in \{1,\ldots,d\}$, observe that
\[
\EE\left\{\prod_{i=1}^d (X_i - \EE[X_i])^{\alpha_i}\right\} = \sum_{\ell_1=0}^{\alpha_1} \dots \sum_{\ell_d=0}^{\alpha_d} \EE\left(\prod_{i=1}^d X_i^{\ell_i}\right) \cdot \prod_{i=1}^d \binom{\alpha_i}{\ell_i} \left(\frac{- n N_i}{N}\right)^{\alpha_i - \ell_i}.
\]
We arrive at the conclusion by applying Theorem~\ref{thm:non.central.moments}.
\end{proof}

\section{Numerical implementation}

The formulas in Theorem~\ref{thm:non.central.moments} and Theorem~\ref{thm:central.moments} can be implemented in \texttt{Mathematica} as follows:
\begin{verbatim}
        NonCentral[\[CapitalNu]_, n_, Nvec_, \[Alpha]_, d_] :=
          Sum[FactorialPower[n, Sum[k[i], {i, 1, d}]] /
            FactorialPower[\[CapitalNu], Sum[k[i], {i, 1, d}]] *
              Product[StirlingS2[\[Alpha][[i]], k[i]] *
                FactorialPower[Nvec[[i]], k[i]], {i, 1, d}], ##]
                  & @@ ({k[#], 0, \[Alpha][[#]]} & /@ Range[d]);

        Central[\[CapitalNu]_, n_, Nvec_, \[Alpha]_, d_] :=
          Sum[Sum[FactorialPower[n, Sum[k[i], {i, 1, d}]] /
            FactorialPower[\[CapitalNu],
              Sum[k[i], {i, 1, d}]]*(-n / \[CapitalNu]) ^
                Sum[\[Alpha][[i]] - ell[i], {i, 1, d}] *
                  Product[Binomial[\[Alpha][[i]], ell[i]] *
                    StirlingS2[ell[i], k[i]] *
                      Nvec[[i]]^(\[Alpha][[i]] - ell[i]) *
                        FactorialPower[Nvec[[i]], k[i]], {i, 1, d}], ##]
                          & @@ ({k[#], 0, ell[#]} & /@ Range[d]), ##]
                            & @@ ({ell[#], 0, \[Alpha][[#]]} & /@ Range[d]);
\end{verbatim}

\section*{Funding}

Ouimet's funding was made possible through a contribution to Christian Genest's research program from the Trottier Institute for Science and Public Policy.

\phantomsection
\addcontentsline{toc}{chapter}{References}

\bibliographystyle{authordate1}
\bibliography{Ouimet_2024_moments_hypergeometric_bib}

\end{document}